\documentclass[11pt,a4paper,reqno]{amsart}
\usepackage{amsmath,amssymb,amsfonts,amsthm}
\usepackage{a4wide}
\usepackage{color}
\usepackage{pdfsync}
\usepackage{graphicx}
\usepackage{hyperref}
\newtheorem{theorem}{Theorem}[section]
\newtheorem{lemma}[theorem]{Lemma}
\newtheorem{proposition}[theorem]{Proposition}

\newtheorem{remark}[theorem]{Remark}
\numberwithin{equation}{section}
\makeatletter

\@addtoreset{equation}{section}
\makeatother

\newcommand{\N}{{\mathbb N}}
\newcommand{\Z}{{\mathbb Z}}

\newcommand{\R}{{\mathbb R}}

\newcommand{\pa}{{\partial}}
\newcommand{\na}{{\nabla}}

\newcommand{\pd}{\partial}

\newcommand{\loc}{{\rm loc}}

\newcommand{\supp}{{\rm supp}\,}

\newcommand{\cal}[1]{\mathcal{#1}} 

\def\curl{{\rm curl }\,}
\def\div{{\rm div }\,}

\def\capa{\mathrm{cap}}

\title[The Two Dimensional  Euler equations on singular exterior domains]{The Two Dimensional  Euler equations \\ on singular exterior domains}
\author[D. G\'erard-Varet \& C. Lacave]{David G\'erard-Varet \& Christophe Lacave}

\address[D. G\'erard-Varet]{Universit\'e Paris-Diderot (Paris 7)\\
Institut de Math\'ematiques de Jussieu - Paris Rive Gauche\\
UMR 7586 - CNRS\\
B\^atiment Sophie Germain \\
Case 7012\\
75205 PARIS Cedex 13\\
France.} \email{gerard-varet@math.jussieu.fr}

\address[C. Lacave]{Universit\'e Paris-Diderot (Paris 7)\\
Institut de Math\'ematiques de Jussieu - Paris Rive Gauche\\
UMR 7586 - CNRS\\
B\^atiment Sophie Germain \\
Case 7012\\
75205 PARIS Cedex 13\\
France.} \email{lacave@math.jussieu.fr}

\date{\today}

\begin{document}
\maketitle
\begin{abstract}
This  paper is a follow-up of article \cite{GV-L}, on the existence of global weak solutions to the two dimensional Euler equations in singular domains.  In \cite{GV-L}, we have established the existence of weak solutions for a large class of bounded domains, with  initial vorticity in $L^p$ ($p>1$). For unbounded domains, we have proved a similar result only  when the initial vorticity is in $L^p_{c}$ ($p>2$) and when the domain is the exterior of a single obstacle. The goal here is to retrieve these two restrictions: we consider general initial vorticity in $L^1\cap L^p$ ($p>1$),  outside  an arbitrary  number of obstacles (not reduced to points).
\end{abstract}


\section{Introduction}

The motion of a 2D ideal incompressible flow in a domain $\Omega$ is governed by the Euler equations:
\begin{eqnarray}
\partial_{t} u + u\cdot \nabla u =-\nabla p && \forall (t,x) \in (0,\infty)\times \Omega, \label{Euler1}\\
\div u =0 && \forall (t,x) \in [0,\infty)\times \Omega,\label{Euler2}\\
 u \cdot \nu =0 && \forall (t,x) \in [0,\infty)\times \partial\Omega, \label{Euler3}\\
 u(0,x) = u^0(x) && \forall x \in \Omega, \label{Euler4}
\end{eqnarray}
where $u=(u_{1}(t,x_{1},x_{2}),u_{2}(t,x_{1},x_{2}))$ is the velocity, $p=p(t,x_{1},x_{2})$ is the pressure and $\nu$ the unit normal vector to $\partial\Omega$ pointing outside the fluid domain.

\medskip
Construction of solutions to these equations has been a constant concern over the last century. Let us mention the pioneering paper of Wolibner \cite{Wolibner} on smooth solutions in smooth bounded domains (see also Kato \cite{Kato} and Temam \cite{Temam75}). Among many others, one can also mention the work of McGrath \cite{McGrath}, resp.  Kikuchi \cite{Kikuchi}  on the existence of smooth solutions in  the whole plane, resp. in smooth 2D exterior domains. Of course, a key feature of the 2D equations is the conservation of vorticity  
 $\omega:=\curl u = \partial_{1}u_{2}-\partial_{2} u_{1}$, which verifies a transport equation:
\begin{equation}\label{Euler5}
\partial_{t} \omega + u\cdot \nabla \omega =0 \qquad \forall (t,x) \in (0,\infty)\times \Omega.
\end{equation}
It propagates bounds on the vorticity, that allows to go from local to global in time solutions,  and to introduce several notions of weak solutions. The most famous result in this direction is certainly the one of Yudovich \cite{Yudo}: it provides existence and uniqueness of global solutions with bounded vorticity, in various domains (with smooth or without boundaries), see also Bardos \cite{Bardos} for another existence proof. Weaker notions of global solutions have been developed since: solutions with $L^p$ vorticity, as introduced by DiPerna and Majda \cite{DiPernaMajda}, or solutions whose vorticy is a signed measure in $H^{-1}$, built by Delort \cite{Delort}. Note that uniqueness of these weaker solutions is unknown. 

\medskip
We stress that in above studies, the domain $\Omega$ is always taken smooth. This is due to mathematical technicalities, such as ensuring the $L^p$ continuity of Riesz transforms over $\Omega$. This smoothness assumption is of course natural in the context of smooth solutions, but it is a big restriction with regards to weak solutions. Indeed, in many situations, the singularities of the flow are created by the flow domain itself, for instance the irregularity of a solid obstacle. Hence, it is highly desirable to extend the theory of weak solutions to general domains. 

\medskip
This issue has been investigated by various authors in the recent years. One can mention for instance the work of M. Taylor \cite{Taylor}, who established the existence of weak solutions for convex and bounded $\Omega$. The convexity assumption is related there to regularity results for the Laplace equation $\Delta \psi = \omega$ in $\Omega$, set with Dirichlet conditions. It is indeed well-known that for convex $\Omega$,  $\omega \in L^2(\Omega) \Rightarrow \psi \in H^2(\Omega)$. This yields the continuity of the Riesz Transform over $L^2$, and so classical methods for  weak solutions can be transfered to this case (see also \cite{BDT} for further refined results in convex and bounded domains).

One can also mention the work \cite{L-IHP} by the second author, that establishes the existence of global solutions with $L^\infty$ vorticity in the exterior of a smooth Jordan curve. This existence result is obtained through an approximation process, considering smooth obstacles shrinking to the curve. It uses several tools of complex analysis, notaby conformal transform, to derive uniform bounds on the sequence of approximate solutions. Note that another related asymptotics, namely obstacles shrinking to a point falls into the scope of a very active area of research, as testified for instance by the recent works \cite{BLM,ILL,L-CPDE,LLL,LionsMasmoudi,Lopes,MikelicPaoli}.

Finally, let us cite the recent papers \cite{BDT,DT,Lac-uni,LMW} on the Euler flow in polygons. In such polygonal domains, one can even go further in the analysis (using elliptic theory in domains with corners), and deal with the uniqueness of Yudovitch solutions.

\medskip
Still, the studies just mentioned have strong limitations on the geometry of the domain $\Omega$.  In this respect, a big step forward  has been made in our recent paper \cite{GV-L}. There, we have proved the existence of global weak solutions with  $L^p$ vorticity  for a large class of non-smooth open sets.  Namely, we have considered two kinds of sets: 
\begin{itemize}
\item Bounded domains: 
$$ \Omega = \widetilde \Omega \setminus \bigcup_{i=1}^k {\cal C}^i $$ 
where $\widetilde \Omega$ is any simply connected open set, and the ${\cal C}^i$'s are disjoint obstacles. 

\smallskip
\item Exterior domains 
$$ \Omega = \R^2 \setminus {\cal C}  $$
for a single obstacle ${\cal C}$.
\end{itemize}
 Let us point out that in \cite[p. 133]{GV-L},  obstacles are defined as  compact connected subsets {\em of positive capacity}. As will be shown below, {\it cf} Proposition \ref{prop:equivalence}, this is the same as simply saying that the connected compact subsets are not reduced to  points.  Hence, the bounded domains considered in \cite{GV-L} are very general. The existence of solutions with $L^p$ vorticity ($p > 1$) is proved using in a crucial way the notion of $\gamma$-convergence of open sets, described extensively in \cite{Henrot} (see \cite[App. C]{GV-L} for a short summary).
   
 However, in the case of exterior domains, the existence result of \cite{GV-L} is not so strong. First, we consider only a single obstacle. Second, we must assume that the vorticity is in $L^p_c(\Omega), p > 2$. These restrictions are due to our method of proof, as will be explained later on.  
 
\medskip
{\it Our goal in this paper is to get rid of these limitations. We shall show the existence of global weak solutions for vorticities in $L^1\cap L^p$ ($p>1$), in the exterior of an arbitrary finite number of  obstacles (not reduced to points).} Again, it will improve \cite{GV-L} because we consider  several obstacles, we do not  assume that the initial vorticity is compactly supported, and we treat indices $p\in (1,2]$. Hence, we will reach the  generality obtained for bounded domains \cite[Theo. 1]{GV-L}.

\bigskip
Let us now state rigorously our result. We consider domains of the type
\begin{equation}\label{Omegatype}
\Omega = \R^2\setminus \Big(\bigcup_{i=1}^k \cal{C}^{i}\Big),\quad k\in \N
\end{equation}
\begin{equation}\label{Omegatype2}
\textrm{where} \quad  \cal{C}^1,\dots,\cal{C}^k \text{ are disjoints connected compact subsets not reduced to points.}
\end{equation}

\medskip
We consider initial data satisfying 
\begin{equation} \label{initialdata}
u^0 \in L^2_{\loc}(\overline{\Omega}), \quad u^0 \rightarrow 0 \:  \mbox{ as } \: |x| \rightarrow +\infty,  \quad \curl u^0 \in L^1\cap L^p(\Omega) 
\end{equation}
for some  $p \in (1,\infty]$.  We want also $u^0$ to be divergence-free and  tangent to the boundary. Nevertheless, due to the irregularity of $\Omega$, this tangency condition has to be expressed in a weak sense. Namely, we ask 
\begin{equation}\label{imperm}
 \int_\Omega u^0 \cdot \na h = 0  \quad \text{for all } h \in H^1_{\loc}(\Omega) \: \text{ such that 
 $\na h \in L^2(\Omega)$ and } h(x) = 0 \: \text{for large $x$}.   \end{equation}
Note that this exactly amounts to the usual divergence-free and tangency condition when $u^0$ and $\Omega$ are smooth: indeed, integrating by parts, we find in such case
$$  \int_\Omega u^0 \cdot \na h = 0 \: = \: \int_{\pa \Omega} u^0 \cdot \nu \, h \: - \: \int_\Omega \div u^0 \,  h, 
$$
this identity being valid for all functions $h\in \cal{D}(\R^2)$. This easily implies   that $\div u^0 = 0$ and $u^0 \cdot \nu \vert_{\pa \Omega} = 0$. We stress that the assumptions \eqref{initialdata}-\eqref{imperm} is not restrictive: we will prove in Subsection \ref{sect conv u0} that for any function $\omega^0\in L^1\cap L^p(\R^2)$ there exists some vector fields $u^0$ verifying \eqref{initialdata}-\eqref{imperm} with $\curl u^0=\omega^0$. Let us also note that such a vector field is not always square integrable at infinity (for $\omega_0$ compactly supported, $u^0$ has a finite energy iff the sum of the circulations around each obstacles is equal to $-\int \omega_0$, which is an important constraint).

\medskip
Similarly to \eqref{imperm}, the weak form of the  divergence free \eqref{Euler2} and tangency conditions \eqref{Euler3} on the Euler solution $u$  reads:  for almost all $t$,
\begin{equation} \label{imperm2}
 \int_\Omega u(t,\cdot) \cdot  \na h  = 0 \quad \text{for all } \: h \in H^1_{\loc}(\Omega) \:  \text{ such that 
 $\: \na h \in L^2(\Omega)$ and } \: h(x) = 0 \: \text{for large $x$}.
\end{equation}
Finally, the weak form  of the momentum equation \eqref{Euler1}  reads:
\begin{equation} \label{Eulerweak}
\mbox{for all } \, \varphi \in {\cal D}\left([0, +\infty) \times \Omega\right) \mbox{ with } \div \varphi = 0, \quad  \int_0^{\infty} \int_\Omega \left( u \cdot \pd_t \varphi +   (u \otimes u) : \na \varphi \right)  = -\int_\Omega u^0 \cdot \varphi(0, \cdot).
\end{equation}
Our main theorem is
\begin{theorem} \label{theoremmain}
Assume that $\Omega$ is of type  \eqref{Omegatype}-\eqref{Omegatype2}. Let $p \in (1,\infty]$ and $u^0$ as in \eqref{initialdata}-\eqref{imperm}. Then there exists  
 $$u \in L^\infty_{\loc}(\R^+; L^2_{\loc} (\overline{\Omega})), \: \mbox{ with } \: \curl u \in   L^\infty(\R^+; L^1\cap L^p (\Omega))$$
which is a global weak solution of \eqref{Euler1}-\eqref{Euler4} in the sense of \eqref{imperm2} and \eqref{Eulerweak}.
\end{theorem}

Moreover, the solution of the theorem has a Hodge-De Rham decomposition, the weak circulations are conserved, and we have the two following estimates:
\begin{equation}\label{finalenergy}
\| \curl u \|_{L^\infty(\R^+,L^p(\Omega))} \leq  \| \curl u^{0} \|_{L^p}, \quad \| \curl u \|_{L^\infty(\R^+,L^1(\Omega))} \leq  \| \curl u^0 \|_{L^1}
\end{equation}
More details will be given in due course.

\begin{remark}
If $p\geq 4/3$, let us stress that  such a solution is also a solution of the vorticity equation \eqref{Euler5} in the sense of distributions, namely,
\begin{equation} \label{vorticityweak}
\mbox{for all } \, \psi \in {\cal D}\left([0, +\infty) \times \Omega\right), \quad  \int_0^{\infty} \int_\Omega \left( \omega  \pd_t \psi +   \omega u \cdot \nabla\psi \right)  = -\int_\Omega \omega^0  \psi(0, \cdot).
\end{equation}
Indeed, for any $\psi \in {\cal D}\left([0, +\infty) \times \Omega\right)$,  $\varphi:=\nabla^\perp \psi$ is a test function for which \eqref{Eulerweak} holds true. As $u$ is uniformly square integrable on the support of $\varphi$ and $\curl u\in L^\infty(\R^+,L^{4/3}(\Omega))$ standard elliptic estimates imply that $u$ belongs to $W^{1,4/3}$ hence to $L^{4}$ on the support of $\varphi$. Then an integration by parts implies \eqref{vorticityweak}.
\end{remark}

\begin{remark}
As in \cite{GV-L}, the main point is that we assume nothing about the regularity of the boundary (the obstacles can  be as exotic as a Koch snowflake). This is made possible by our method of proof, based on the $\gamma$-convergence theory: it only requires the velocity to be in $L^2$ near the boundary. Of course, it is also because we deal only with existence issues. Uniqueness of weak solutions requires in general   more regularity (a uniform bound of $\|u\|_{W^{1,p}}/(p\ln p)$ for large $p$ or a log-lipschitz estimate up to the boundary) which cannot be obtained without some stronger assumptions on the domains. Even in bounded domains, Jerison and Kenig \cite{Kenig} exhibited an example of $\omega$ smooth, $\partial\Omega\in C^1$ and where $Du$ is not integrable (where $u=\nabla^\perp\psi$ with $\psi$ solution of the Dirichlet problem $\Delta\psi=\omega$). Hence to prove the uniqueness, the authors in  \cite{BDT,DT,Lac-uni,LMW} have to assume that $\partial \Omega$ is smooth, except in a finite number of points.
\end{remark}
Let us eventually comment on the proof of Theorem \ref{theoremmain}. The basic idea is to construct a sequence of approximate smooth domains $\Omega_n$ and initial data $u^0_n$, that generate smooth solutions $u_n = \na^\perp \psi_n$. The point is to derive uniform bounds on $\psi_n$, starting from the $L^p$ uniform bound on the vorticity $\omega_n = \Delta \psi_n$. In \cite{GV-L}, we used  a sort of  Poincar\'e inequality on large balls $B(0,R)$ containing the obstacle. There, the fact that the initial vorticity was supported in such a  ball was crucial. Furthermore,  it was  necessary to show that the compact support of $\omega_n$ was controlled uniformly in $n$ at later times.  We relied here on an explicit representation of $\psi_n$ in terms of $\omega_n$. This representation, valid only for a single obstacle, involved a biholomorphism $\cal{T}_n$ sending the exterior of $\Omega_n$ to the exterior of the unit disk. The fact that the vorticity was zero at infinity  (and in $L^p$, $p > 2$ !) was again useful, to handle this representation formula. The originality in the present work is to produce an alternative proof for the uniform bound. Loosely, we bring together several tools developed in \cite{GV-L} (uniform Poincar\'e inequality in the exterior of one obstacle, Caratheodory convergence) with a new  inversion argument inspired from \cite{LLL}.

\medskip
The outline of the paper is as follows. In Section \ref{sect 2}, we explain how to approximate $\Omega$ and $u^0$ by some smooth $\Omega_n$ and $u^0_n$. Section \ref{sect 3} is the central one: it shows how to obtain  estimates on the Euler  approximations $u_n$, uniformly in $n$. This allows to complete the proof of the existence of weak solutions in $\Omega$, through compactness arguments (Section \ref{sect 4}).

\section{Geometry and initial data approximation}\label{sect 2}

We begin this section by recalling that a domain of type \eqref{Omegatype}-\eqref{Omegatype2} can be approximated in the Hausdorff topology by smooth domains. 

\begin{proposition}{\cite[Prop. 1]{GV-L}}\label{Om conv}
Let $\Omega$ be of type \eqref{Omegatype}-\eqref{Omegatype2}. Then $\Omega$ is the limit of a sequence
\[
\Omega_{n}:= \R^2 \setminus \Big( \bigcup_{i=1}^k \overline{O_{n}^{i}}\Big)
\]
where the $O_n^i$ are  smooth Jordan domains, with  $\overline{O_n^i}$ converging to $\cal{C}^i$ in the  Hausdorff topology.
\end{proposition}

A short reminder about Hausdorff topology can be found in \cite[App. B]{GV-L}. For our analysis, we just recall here this useful property: 
\begin{equation}\label{hausdorff}
\begin{split}
\text{for any compact set $K\subset \Omega$, there exists $n_{K}>0$ such that }
K\subset \Omega_{n}, \  \forall n\geq n_{K}.
\end{split}
\end{equation}

Another geometrical feature which turns out to be crucial in our $L^2$- framework is the positive Sobolev capacity of the obstacles. The Sobolev $H^1$ capacity of a compact set $E \subset \R^2$ is defined by 
$$ {\rm cap}(E) \: := \: \inf \{ \| v \|^2_{H^1(\R^2)}, \: v \ge 1 \: \mbox{ a.e.    in a neighborhood of } E\},  
$$
with the convention that ${\rm cap}(E)= +\infty$ when the set at the r.h.s. is empty. We refer to  \cite{Henrot} for an extensive study of this notion, while the basic properties are listed in \cite[App. A]{GV-L}. In particular we recall that a point has zero capacity, whereas the capacity of a smooth Jordan arc is positive.

\medskip
By \eqref{Omegatype2}, our obstacles ${\cal C}^i$ are compact,  connected, and not reduced to a point. This is enough to ensure that they have positive capacity.  This is expressed by the following
\begin{proposition} \label{prop:equivalence}
Let $\cal{C}$ be a connected compact subset of $\R^2$. Then, we have the following equivalence:
\[
\capa (\cal{C})>0 \qquad \text{if and only if}\qquad  \cal{C} \text{ is not reduced to one point}.
\]
\end{proposition}
\begin{proof}
It is already well known that the capacity of a single point is zero. Conversely, assume that 
 $\cal{C}$ is a connected compact set with at least two different points $x\neq y$.  Let $\Pi_{[x,y]}({\cal C})$ the projection of 
 $\cal{C}$ over the segment $[x,y]$. It is connected, being the continuous image of ${\cal C}$, and contains $x$ and $y$. Thus, 
 $\Pi_{[x,y]}({\cal C}) = [x,y]$, and it has in particular Hausdorff dimension 1. As the projection $\Pi_{[x,y]}$ is  Lipschitz, and as Lipschitz functions decrease the Hausdorff dimension, we deduce that the Hausdorff dimension of ${\cal C}$ is greater than $1$. This implies that the capacity of ${\cal C}$ is positive, see \cite{Evans}. 
\end{proof} 
The positive  capacity of our obstacles is important notably  because it provides a uniform Poincar\'e  inequality for functions that vanish at the (regularized) obstacles. More precisely:

\begin{lemma}{\cite[Lem. 1]{GV-L}}\label{Poincare}
Let $\cal{C}$ be a connected compact subset of $\R^2$ not reduced to a point, and let $\overline{O_{n}}$ the closures of smooth Jordan domains,  converging to $\cal{C}$ in the Hausdorff sense. For any $\rho>0$ such that $\cal{C}\subset B(0,\rho)$, there exists $C_{\rho}$ and $N_{\rho}$ such that
\[
\| \varphi \|_{L^2(B(0,\rho)\setminus O_{n})} \leq C_{\rho} \| \nabla \varphi \|_{L^2(B(0,\rho)\setminus O_{n})}, \ \forall \varphi \in C^\infty_{c}( \R^2 \setminus \overline{ O_{n}}), \ \forall n\geq N_{\rho}.
\]
\end{lemma}

Now, we focus on  the initial data approximation. Let $u^0$ verifying \eqref{initialdata}-\eqref{imperm}. We write $\omega^{0}:= \curl u^0$ and consider after truncation and convolution a sequence $\omega^0_{n}$ such that
\begin{equation}\label{approx om}\begin{split}
\omega^0_{n} \in C^\infty_{c}(\Omega_{n}), \   \| \omega^0_{n} \|_{L^p(\Omega_{n})}\leq \| \omega^0\|_{L^p(\Omega)},  \  \omega^0_{n} \to \omega^0 \textrm{ strongly in } & L^1(\R^2) \cap L^p(\R^2) \: \textrm{ for finite $p$}, \\
\textrm{  in } & L^q(\R^2) \: \textrm{  for all finite $q$ if $p=\infty$}.
\end{split}\end{equation}
Here, the functions defined on $\Omega_{n}$ are implicitly extended by zero, so that convergence results are stated in $\R^2$. Let us  note that the assumption $\omega^0_{n} \in C^\infty_{c}(\Omega_{n})$ is easy to achieve thanks  to \eqref{hausdorff}: any function compactly supported in $\Omega$ is compactly supported in $\Omega_{n}$ for $n$ large enough.

When the domain is not simply connected, the vorticity is not sufficient to determine uniquely the velocity: we also need to specify  the circulations around the obstacles. Due to our irregular domain $\Omega$, we need to define  these circulations in a weak sense. Therefore, we introduce some smooth cutoff functions $\chi^{i,\varepsilon}$ such that
\[
\chi^{i,\varepsilon} \equiv 1 \text{ on } \cal{C}^{i,\varepsilon},\quad  \chi^{i,\varepsilon} \equiv 0 \text{ on } \R^2\setminus\cal{C}^{i,2\varepsilon}, 
\]
where $\cal{C}^{i,\varepsilon}:=\{ x, \ d(x,\cal{C}^i)\leq \varepsilon\}$. Again by \eqref{hausdorff}, we can fix $\varepsilon$ and $n_{0}(\varepsilon)$ such that for all $n\geq n_{0}(\varepsilon)$:
\[
\chi^{i,\varepsilon} \equiv 1 \text{ on } \overline{O^{i}_{n}},\quad  \chi^{i,\varepsilon} \equiv 0 \text{ on } \overline{O^{j}_{n}}\ \forall j\neq i.
\]
For brevity, we drop the subscript $\varepsilon$. Then, for any $v \in L^1_{\loc}(\Omega)$ with $\curl v \in L^1_{\loc}$,  we define the weak circulation of $v$ around $\cal{C}^i$ by:
\begin{equation}\label{weak circu}
\gamma^i(v) := -\int_{\Omega} \chi^i \curl v - \int _{\Omega}  v \cdot \nabla^\perp \chi^i. 
\end{equation}
Let us note that this definition is independent of the choice of $\chi^i$. Also, when $\Omega$ and $v$ are smooth enough, this weak definition coincides with the standard definition of the circulation around ${\cal C}^i$: 
$$\displaystyle \gamma^i(v) =  \oint_{\pa {\cal C}^i} v \cdot \tau.$$
The above integral is considered in the counter clockwise sense, hence $\tau = -\nu^\perp$ for $\nu$ the normal vector pointing inside the obstacles.

Finally, we remark that   by our regularity assumption on $u^0$, $\gamma^i(u^0)$ is well-defined.  
It is now a classical result that, given any vorticity $\omega_0^n \in C^\infty_c(\Omega_n)$ and   real numbers $\gamma^i$, $i=1\dots k$, there exists a unique divergence free and tangent vector field $u^0_n$ over $\Omega_n$ such that $\curl u^0_n = \omega^0_n$ and $\oint_{\partial O_{n}^i} u^0_{n}\cdot \tau\,ds = \gamma^i(u^0)$ for all $i$.

Actually, we will prove in Subsection \ref{sect conv u0} that this kind of characterization also holds for more irregular domains and data:  for a given $\displaystyle \omega^0\in L^1\cap L^p(\Omega)$ and $\gamma\in \R^k$,  there exists a unique $u^0$ verifying \eqref{initialdata}-\eqref{imperm} such that 
\[
\curl u^0 = \omega^0 \text{ in } \cal{D}'(\Omega), \quad \gamma^i(u^0) = \gamma^i‚ \ \forall i=1\dots k.
\]

\medskip
After these preliminary considerations, the main point will be to show that the Euler flow $u_n$ generated by  $u^0_n$ in $\Omega_n$ converges to an Euler flow $u$ generated by $u^0$ in $\Omega$. More precisely,  our main theorem concerning existence of global weak solution for initial velocity $u^0$ (resp. $\omega^0$, $\gamma$) will be a direct consequence of the following stability result:
\begin{theorem}\label{theo stability}
Assume that $\Omega$ is of type  \eqref{Omegatype}-\eqref{Omegatype2}. Let $p \in (1,\infty]$ and $u^0$ as in \eqref{initialdata}-\eqref{imperm} (resp. let $\omega^0\in L^1\cap L^p(\Omega)$ and $\gamma\in \R^k$). For any sequences:
\begin{itemize}
\item[a)] $\Omega_{n}$ of smooth domains (as in Proposition \ref{Om conv}) converging to $\Omega$;
\item[b)] $\omega^0_{n} \in C^\infty_{c}(\Omega_{n})$, uniformly bounded in $L^p(\R^2)$ such that  $\omega^0_{n} \rightarrow \omega^0=\curl u^0$ strongly in  $L^1(\R^2)$;
\item[c)] $\gamma_{n}\in \R^k$ such that $ \gamma_{n} \rightarrow \gamma(u^0)$;
\end{itemize}
we consider the unique strong solution $(u_{n},\omega_{n}=\curl u_{n})$ of the Euler equations on $\Omega_{n}$ with initial vorticity $\omega^0_{n}$ and initial circulations $\gamma_{n}$ (Kikuchi \cite{Kikuchi}). Then we can extract a subsequence such that
\begin{itemize}
\item[1)] $\omega_{n} \rightharpoonup \omega$ weak-$*$ in $L^\infty((0,\infty);L^q(\R^2))$ for any $q\in [1,p]$;
\item[2)] $u_{n} \rightharpoonup u$ weak-$*$ in $L^\infty((0,\infty);L^2_{\loc}(\overline{\Omega}))$;
\item[3)] $\omega = \curl u$ in $\cal{D}'(\R_{+}\times\Omega)$ and $u$ is a global weak solution of \eqref{Euler1}-\eqref{Euler4} in the sense of \eqref{imperm2} and \eqref{Eulerweak} with initial data $u^0$ (resp. $\omega^0, \gamma$).
\end{itemize}
\end{theorem}
The convergences are considered by extending $u_{n}$ and $\omega_{n}$ by zero on $\cup O_{n}^i$.
Of course, if $\Omega$ is smooth and $p=\infty$, the uniqueness result of Yudovich \cite{Yudo} allows us to state that any subsequence converges to the unique weak solution, hence the previous theorem holds without extraction of a subsequence. 

\medskip
The two following sections are dedicated to the proof of Theorem \ref{theo stability}.

\section{Uniform estimates} \label{sect 3}

Let  $\Omega_{n}$, $\omega_{n}^0$ and $\gamma_{n}$ verifying the assumptions of Theorem \ref{theo stability}. By results of Kikuchi \cite{Kikuchi}, there exists  a unique global strong solution $(u_{n},\omega_{n})$ of the Euler equations \eqref{Euler1}-\eqref{Euler4}. In this smooth setting, the transport of vorticity  through equation \eqref{Euler5} implies that
\begin{itemize}
\item  the  $L^q$ norm of the vorticity is conserved:
\begin{equation}\label{om cons}
\| \omega_n(t,\cdot) \|_{L^q(\Omega_{n})} = \|\omega_{n}^0 \|_{L^q(\Omega_{n})}, \ \forall q\in [1,\infty],\ \forall t\in \R^+ ;
\end{equation}
\item the mass of the vorticity is conserved:
\begin{equation}\label{om cons2}
\int_{\Omega_{n}}\omega_n(t,\cdot)  = \int_{\Omega_{n}} \omega_{n}^0 , \ \forall t\in \R^+ .
\end{equation}
\end{itemize}
Moreover, the Kelvin's theorem gives the conservation of the circulation:
\begin{equation}\label{circ cons}
\oint_{\partial O_{n}^i }u_n(t,\cdot)\cdot \tau \, ds  =\gamma_{n}^i , \ \forall t\in \R^+,\ \forall i=1,\dots,k.
\end{equation}

\medskip
By assumption b) in Theorem \ref{theo stability} and by \eqref{om cons}, we  have easily that
\begin{equation}\label{om est}
\omega_{n} \text{ is uniformly bounded in }L^\infty(\R^+;L^1\cap L^p(\R^2)).
\end{equation}

\subsection{Biot-Savart decomposition}

We reconstruct here the velocity in terms of the vorticity and the circulations. By results related to the Hodge-De Rham theorem, there exists a unique vector field $u_{n}(t,\cdot)$ which is tangent to the boundary, divergence free, tending to zero at infinity and whose  curl is $\omega_{n}(t,\cdot)$ and circulations $\gamma_{n}$:
\begin{equation}\label{BS}
u_{n}(t,x) = \nabla^\perp \psi_{n}^0(t,x) + \sum_{i=1}^k \alpha_{n}^i(t)  \nabla^\perp \psi_{n}^i(x).
\end{equation}
In this decomposition, $\psi_{n}^0$ satisfies the Dirichlet problem:
\[
\Delta \psi_{n}^0 = \omega_{n} \text{ in }\Omega_{n},\quad \psi_{n}^0\vert_{\partial \Omega_{n}}=0,\quad \psi_{n}^0(x) =\cal{O}(1) \text{ as }x\to \infty
\]
whereas $\psi_{n}^i$ are harmonic functions, that satisfy:
\begin{multline*}
\Delta \psi_{n}^i = 0 \text{ in }\Omega_{n},\  \partial_{\tau} \psi_{n}^i \vert_{\partial \Omega_{n}}=0, \  \oint_{\pa {\cal O^j_n}} \nabla^\perp \psi_{n}^i  \cdot \tau =\delta_{i,j} \text{ for }j=1,\dots,k, \\
  \psi_{n}^i(x) =\frac1{2\pi}\ln |x| + \cal{O}(1) \text{ as }x\to \infty.
\end{multline*} 
As this vector field can be chosen up to a constant, we can further assume  that
\begin{equation*}
\psi_{n}^i=0 \quad \text{on} \quad \partial O_{n}^i.
\end{equation*}
In particular,  the circulation conditions on the $\psi^i_n$ ($i=1,\dots,k$) together with \eqref{circ cons} lead to: 
\begin{equation} \label{defalphain}
\alpha_{n}^i(t) = \gamma^i_n -  \int_{\pa O^i_n}  \na^{\perp} \psi^0_n \cdot \tau, 
\end{equation}
where the circulation at the r.h.s. can be expressed in a weak form similar to  \eqref{weak circu}:
\begin{equation}\label{circu}
\int_{\pa O^i_n}  \na^{\perp} \psi^0_n \cdot \tau = -\int_{\Omega_n} \chi^i \omega_{n} - \int _{\Omega_n}  \na \psi^0_n \cdot \nabla \chi^i \: = \: -\int_{\R^2} \chi^i \omega_{n} - \int _{\R^2}  \na \psi^0_n \cdot \nabla \chi^i.
\end{equation}

\begin{remark}\label{rem infinity}
The behavior at infinity of the stream functions is somehow classical. It  can be deduced from the link between 2D harmonic vector fields and holomorphic functions. Namely,   if $v = \na^\perp \psi$ for a harmonic stream function  $\psi$ in an open set $U$, then the mapping 
$$f \: : \:  z = x + i y \: \mapsto \:  v_1(x,y) - i v_2(x,y) \: (=-\pa_y \psi(x,y)  - i \pa_x \psi(x,y))$$ 
is holomorphic in $U$. The behaviour at infinity of $\psi$ follows from the Laurent expansion of $f$ at $z=\infty$.   We refer to \cite{Kikuchi} for more details on $\psi_{n}^i$, $i \ge 1$. As regards $\psi_{n}^0$, the fact that $\omega_{n}$ is compactly supported  implies that $\psi_{n}^0$ is also harmonic outside a disk. Moreover,  the associated holomorphic function admits a Laurent expansion at infinity  whose  first term is $\mathcal{O}(1/z^2)$, because $\nabla \psi_{n}^0\in L^2(\Omega_{n})$ (Lax-Milgram). We note here that if the  vorticity is no longer compactly supported, the behavior at infinity is far less clear. It explains the difficulties to be met in Subsection \ref{sect conv u0}.
\end{remark}

In the case of only one obstacle ($k=1$) treated in \cite{GV-L}, some explicit formula could be used. It involved  the unique Riemann mapping ${\cal T}_{n}$ which sends the exterior of $O_{n}^1$ to the exterior of the unit disk  and satisfies 
$$\mathcal{T}_n(\infty)=\infty, \quad \mathcal{T}'_n(\infty) > 0 \:  \text{(meaning that  at infinity: $\mathcal{T}_n(z) = \beta z + \mathcal{O}(1)$ with $\beta\in \R^+$)}. $$
In this case, we could write 
\begin{equation}\label{BS formula}
\psi_{n}^0(t,x) = \frac1{2\pi}\int_{(O_{n}^1)^c} \ln\frac{|{\cal T}_{n}(x)-{\cal T}_{n}(y)|}{|{\cal T}_{n}(x)-{\cal T}_{n}(y)^\ast| |{\cal T}_{n}(y)|}\omega_{n}(t,y)\, dy,\quad \psi_{n}^1(t,x)= \frac1{2\pi}\ln |{\cal T}_{n}(x)|,
\end{equation}
with notation $z^\ast=\frac{z}{|z|^2}$. Moreover, we could use  in \cite{GV-L}  a result related to the Caratheodory theorem on  the convergence of $\mathcal{T}_n$:  

\begin{proposition}{\cite[Prop. 15]{GV-L}}\label{Caratheo}
Let $\Pi$ be the unbounded connected component of $\R^2\setminus \mathcal{C}^1$. There is a unique biholomorphism $\mathcal{T}$ from $\Pi$ to $(\overline{B(0,1)})^c$, satisfying $\mathcal{T}(\infty)=\infty$, $\mathcal{T}'(\infty) > 0$. Moreover, as $\overline{O_{n}^1}$ converges to $\cal{C}^1$ (in the Hausdorff sense), one has  the following convergence properties:  
\begin{itemize}
\item[i)] $\mathcal{T}_n^{-1}$ converges uniformly locally to $\mathcal{T}^{-1}$ in $(\overline{B(0,1)})^c$. 
\item[ii)] $\mathcal{T}_n$ (resp. $\mathcal{T}_n'$) converges uniformly locally to $\mathcal{T}$ (resp. to $\mathcal{T}'$) in $\Pi$. 
\item[iii)] $|\mathcal{T}_n|$ converges uniformly locally to $1$ in $\Omega\setminus\Pi$.
\end{itemize}
\end{proposition}

This proposition together with the explicit formula \eqref{BS formula} was the key in \cite{GV-L} to get uniform estimates on the velocity. A problem that we solve below is  to extend such estimates to the case of several obstacles: $k>1$. 

\subsection{Harmonic part}

Let us fix $i$ in $\{1,\dots,k\}$, and let us  look for  local uniform estimates on $\psi_{n}^i$. Let $K\Subset \Omega$.  Property \eqref{hausdorff} states that there exists $n_{K}$ such that $K\Subset \Omega_{n}$ for any $\displaystyle n\geq n_{K}$.

Denoting by ${\cal T}_{n}^i$ the unique biholomorphism from $\displaystyle (\overline{O_{n}^i})^c$ to the exterior of the unit disk satisfying $\mathcal{T}_n^i(\infty)=\infty$, $\displaystyle (\mathcal{T}_n^i)'(\infty) > 0$,  we infer from \eqref{BS formula} that
\[
\tilde \psi_{n}^i(x) \: := \: \frac1{2\pi} \ln | {\cal T}_{n}^i(x)|
\]
verifies
\[
\Delta \tilde \psi_{n}^i = 0 \text{ in }(O_{n}^i)^c,\quad \tilde \psi_{n}^i \vert_{\partial O_{n}^i}=0, \quad \oint_{\pa O_{n}^i} \nabla^\perp \tilde\psi_{n}^i \cdot \tau =1, \quad \tilde\psi_{n}^i(x) =\frac1{2\pi}\ln |x|+ \cal{O}(1) \text{ as }x\to \infty.
\]
By this explicit formula  and thanks to Proposition \ref{Caratheo}, it is obvious that we have for any compact subset $K'$ of $\Omega$:
\begin{equation}\label{est tilde psi}
\tilde\psi_{n}^i \text{ and }\nabla \tilde\psi_{n}^i \text{ are bounded uniformly in } x\in K',\ n\geq n_{K'}.
\end{equation}

Next, we introduce $\displaystyle \tilde \chi^i:= 1-\sum_{j\neq i} \chi^j$ (with $\chi^j$ defined in  \eqref{weak circu}), which vanishes in a small neighborhood of $O_{n}^j$ for $j\neq i$. Then, we define
\[
\hat \psi_{n}^i:= \psi_{n}^i - \tilde\psi_{n}^i \tilde\chi^i.
\]
As $\partial_{\tau}  \psi_{n}^i =0$ on $\partial \Omega_{n}$, it follows that $\hat \psi_{n}^i$ is constant on $\partial O_{n}^j$ for any $j=1,\dots ,k$ (in particular $\hat \psi_{n}^i=0$ on $\partial O_{n}^i$),   $\displaystyle \Delta \hat \psi_{n}^i = -2\nabla  \tilde\psi_{n}^i \cdot \nabla \tilde\chi^i - \tilde\psi_{n}^i \Delta \tilde\chi^i$, $\displaystyle \oint_{\partial O_{n}^j} \nabla^\perp  \tilde\psi_{n}^i \cdot \tau \, ds =0$ for all $j=1,\dots,k$, $\hat \psi_{n}^i = \mathcal{O}(1)$ and $\nabla \hat \psi_{n}^i = \mathcal{O}(\frac1{|x|^2})$ at infinity. Thanks to these properties, we can perform an energy estimate:
\[
\| \nabla \hat \psi_{n}^i \|_{L^2(\Omega_{n})}^2 =- \int_{\Omega_{n}} \hat \psi_{n}^i \, \Delta \hat \psi_{n}^i - \sum_{j=1}^k \hat \psi_{n}^i \oint_{\partial O_{n}^j}  \nabla^\perp\hat \psi_{n}^i\cdot \tau \, ds=- \int_{\Omega_{n}} \hat \psi_{n}^i \Delta \hat \psi_{n}^i \leq C \|  \hat \psi_{n}^i \|_{L^2({\rm supp}\ \nabla \tilde\chi^i)} 
\]
where we have used \eqref{est tilde psi} with $K'={\rm supp}\ \nabla \tilde\chi^i$. Thanks to the Dirichlet condition on $\hat \psi_{n}^i$ at $\partial O_{n}^i$, we can use Lemma \ref{Poincare} with $\rho$ such that ${\rm supp}\ \nabla \tilde\chi^i\subset B(0,\rho)$, hence
\begin{equation}\label{est nabla hat psi}
\| \nabla \hat \psi_{n}^i \|_{L^2(\Omega_{n})} \text{ is bounded  uniformly in $n$}.
\end{equation}
Applying again Lemma \ref{Poincare} with $\rho$ such that $K\subset B(0,\rho)$, we deduce that
\[
\|  \hat \psi_{n}^i \|_{H^1( B(0,\rho))} \text{ is bounded  uniformly in $n$},
\]
which implies with \eqref{est tilde psi} that:
\[
\|   \psi_{n}^i \|_{H^1(K)} \text{ is bounded  uniformly in $n$}.
\]

This ends the proof of the local estimate of $ \psi_{n}^i$ away from the boundary:
\begin{equation} \label{uniformboundpsini}
\psi_{n}^i \text{ belongs to }H^1_{\loc}(\Omega)\text{ uniformly in $n$ and $i=1,\dots,k$.}
\end{equation}

Eventually, let us show a uniform $H^1$ bound on $\psi_{n}^i$ near the boundary. To this end, we   introduce a smooth function $\chi$  with compact support, such that  $\chi = 1$ in a neighborhood of all obstacles ${\cal C}^i$.  Then, we consider the function $\chi \, \psi_n^i$, $n$ large enough.  It satisfies 
$$ \Delta (\chi \, \psi^i_n) \: = \:  2  \na \chi \cdot \na \psi^i_n + (\Delta \chi) \, \psi^i_n, $$
from which we deduce 
$$ \int_{\Omega_n} | \na (\chi \psi^i_n)|^2 \: \le \:  \| 2 \na \chi \cdot \na \psi^i_n + \Delta \chi \, \psi^i_n  \|_{L^2 } \, 
\|  \chi \, \psi^i_n \|_{L^2}. $$
Note that there is no  boundary term at $\pa \Omega_n$: indeed, we have $\chi = 1$ in a neighborhood of $\pa \Omega_n$, and 
$$  \int_{\pa \Omega_n} (\chi \psi^i_n) \,  \tau \cdot \na^\perp ( \chi \psi^i_n)   =    \sum_j \psi^i_n\vert_{\pa O^j_n}  \int_{\pa O^j_n}  \tau \cdot \na^\perp \psi^i_n   = 0,$$
by the zero circulation around $O^j_n$ $j\neq i$ and by the Dirichlet condition on $\partial O^i_n$.
Moreover, in the inequality above, the first factor at the r.h.s. is supported away from the boundary of $\Omega_n$.  It is therefore bounded uniformly in $n$, by \eqref{uniformboundpsini}. Extending $\psi^i_n$ inside all obstacles $O^j_n$ by their (constant) values at $\pa O^j_n$ we can apply  Lemma \ref{Poincare} to the second factor to state that $\|  \chi \, \psi^i_n \|_{L^2(\Omega_{n})}\leq \|  \chi \, \psi^i_n \|_{L^2((O_{n}^1)^c)}\leq C  \|  \nabla (\chi \, \psi^i_n) \|_{L^2(\Omega_{n})}$
and we end up with 
$$  \|  \nabla (\chi \, \psi^i_n) \|_{L^2(\Omega_{n})}  \: \le \: C  $$
This yields a uniform control of $\na \psi^i_n$ in $L^2$ in a vicinity of the obstacles, and still by Lemma \ref{Poincare}, a uniform control of $\psi^i_n$ in $H^1$ in a vicinity of the obstacles. Combining with  \eqref{uniformboundpsini}, we get that 
$$ \psi_n^i \text{ is bounded  uniformly in $n$ in } H^1_{\loc}(\overline{\Omega}), \quad i=1,\dots, k.$$  

Actually, if we consider the extension of $\psi^i_n$ inside all obstacles $O^j_n$ by their (constant) values at $\pa O^j_n$, Lemma \ref{Poincare} gives that 
$$ \psi_n^i \text{ is bounded  uniformly in $n$ in } H^1_{\loc}(\R^2), \quad i=1,\dots, k.$$

\subsection{Kernel part}

As before,  to get estimates on $\psi_{n}^0$ in the formula \eqref{BS}, we introduce the similar problem in the case of only one obstacle (see \eqref{BS formula}): let $\tilde \psi_{n}^0$ defined by
\begin{equation}\label{tilde-psi}
\tilde\psi_{n}^0(t,x) := \frac1{2\pi}\int_{(O_{n}^1)^c} \ln\frac{|{\cal T}_{n}(x)-{\cal T}_{n}(y)|}{|{\cal T}_{n}(x)-{\cal T}_{n}(y)^\ast| |{\cal T}_{n}(y)|} \omega_{n}(t,y)\, dy
\end{equation}
which verifies
\[
\Delta \tilde\psi_{n}^0 =  \omega_{n} \text{ in }(O_{n}^1)^c,\quad \tilde\psi_{n}^0\vert_{\partial O_{n}^1}=0,\quad \tilde\psi_{n}^0(x) =\cal{O}(1) \text{ as }x\to \infty.
\]

\subsubsection{Uniform estimates} We do not assume that $\omega^0$ is compactly supported and that $p>2$ (case studied in \cite{GV-L}).

In this part, we establish a local uniform estimate of $\tilde\psi_{n}^0$ far away the boundary, thanks to the explicit formula \eqref{tilde-psi} and Proposition \ref{Caratheo}.

\begin{lemma}\label{uniform BS} For any compact subset $K'$ of $\R^2\setminus {\cal C}^1$, we have
\[
 \tilde \psi_{n}^0  \text{ is bounded  uniformly in $n$, $t$ and }x\in K'.
\] 
\end{lemma}
\begin{proof}
For $K'$ fixed, there exists $n_{K'}$ such that $K'\subset \R^2\setminus \overline{O_{n}^1}$ (see \eqref{hausdorff}). We will prove the uniform estimate thanks to the formula \eqref{tilde-psi}.

First, we recall that $|{\cal T}_{n}(y)^\ast|\leq 1$ whereas ${\cal T}_{n}$ converges uniformly to ${\cal T}$ on $K'$ (see Proposition \ref{Caratheo}). Moreover, there exists $C_{1} =C_{1}({\cal T}, K')$ such that $1+1/C_{1} \leq|{\cal T}(x)|\leq 1+C_{1}$ on $K'$. Therefore, there exists $N$ such that we have for all $x\in K'$, $n\geq N$:
\[
|{\cal T}_{n}(x)-{\cal T}_{n}(y)^\ast| \leq |{\cal T}_{n}(x)|+|{\cal T}_{n}(y)^\ast|\leq 1+2C_{1}+1
\]
and
\[
|{\cal T}_{n}(x)-{\cal T}_{n}(y)^\ast| \geq |{\cal T}_{n}(x)|- |{\cal T}_{n}(y)^\ast|\geq 1+\frac{1}{2C_{1}}-1
\]
and in turn
\begin{eqnarray*}
\frac1{2(1+C_{1})}\Big(1- \frac{|{\cal T}_{n}(x)|}{ |{\cal T}_{n}(y)|}\Big) \leq &\dfrac{|{\cal T}_{n}(x)-{\cal T}_{n}(y)|}{|{\cal T}_{n}(x)-{\cal T}_{n}(y)^\ast| |{\cal T}_{n}(y)|} &\leq2C_{1}\Big(1+ \frac{|{\cal T}_{n}(x)|}{ |{\cal T}_{n}(y)|}\Big)\\
\frac1{2(1+C_{1})}\Big(1- \frac{1+2C_{1}}{ |{\cal T}_{n}(y)|}\Big) \leq &\dfrac{|{\cal T}_{n}(x)-{\cal T}_{n}(y)|}{|{\cal T}_{n}(x)-{\cal T}_{n}(y)^\ast| |{\cal T}_{n}(y)|} &\leq 2C_{1}\Big(1+ \frac{1+2C_{1}}{1}\Big)
\end{eqnarray*}
for all $x\in K'$, $n\geq N$. We split $(O_{n}^1)^c$ in three parts: 
$$
A_{1}:={\cal T}_{n}^{-1}(B(0,1+1/(4C_{1}))\setminus B(0,1)),\ A_{2}:={\cal T}_{n}^{-1}(B(0,2(1+2C_{1}))\setminus B(0,1+1/(4C_{1}))),$$
$$ A_{3}:={\cal T}_{n}^{-1}(B(0,2(1+2C_{1}))^c).
$$
In the third subdomain, it is obvious that $|{\cal T}_{n}(y)|\geq 2(1+2C_{1})$ and then
\[
 \Big| \int_{A_{3}} \ln\frac{|{\cal T}_{n}(x)-{\cal T}_{n}(y)|}{|{\cal T}_{n}(x)-{\cal T}_{n}(y)^\ast| |{\cal T}_{n}(y)|} \omega_{n}(t,y)\, dy\Big| \leq C\| \omega_{n}(t,\cdot)\|_{L^1}
 \]
 for any $x\in K'$ and $n\geq N$. In the first subdomain, it is clear that $|{\cal T}_{n}(x)-{\cal T}_{n}(y)|\geq |{\cal T}_{n}(x)|-|{\cal T}_{n}(y)|\geq 1/(4C_{1})$ for all $n\geq N$, hence 
\[
\frac1{2(1+C_{1})}\frac{1/(4C_{1})}{1+1/(4C_{1})} \leq \dfrac{|{\cal T}_{n}(x)-{\cal T}_{n}(y)|}{|{\cal T}_{n}(x)-{\cal T}_{n}(y)^\ast| |{\cal T}_{n}(y)|} \leq 4C_{1}(1+C_{1})
\]
which implies that
\[
 \Big| \int_{A_{1}} \ln\frac{|{\cal T}_{n}(x)-{\cal T}_{n}(y)|}{|{\cal T}_{n}(x)-{\cal T}_{n}(y)^\ast| |{\cal T}_{n}(y)|} \omega_{n}(t,y)\, dy\Big| \leq C\| \omega_{n}(t,\cdot)\|_{L^1}.
\]
In the second subdomain, we note that
\[
\Big|  \ln\frac{|{\cal T}_{n}(x)-{\cal T}_{n}(y)|}{|{\cal T}_{n}(x)-{\cal T}_{n}(y)^\ast| |{\cal T}_{n}(y)|} \Big| \leq C + \Big|\ln{|{\cal T}_{n}(x)-{\cal T}_{n}(y)|}\Big|
\]
so
\[
 \Big| \int_{ A_{2}} \ln\frac{|{\cal T}_{n}(x)-{\cal T}_{n}(y)|}{|{\cal T}_{n}(x)-{\cal T}_{n}(y)^\ast| |{\cal T}_{n}(y)|} \omega_{n}(t,y)\, dy\Big| \leq C\| \omega_{n}(t,\cdot)\|_{L^1}+\| \omega_{n}(t,\cdot)\|_{L^p} \| \ln |{\cal T}_{n}(x)-{\cal T}_{n}(y)| \|_{L^{p'}(A_{2})}
\]
with $p'$ the conjugate of $p$: $1=\frac1p+\frac1{p'}$. The last norm is easy to estimate by changing variable $\xi={\cal T}_{n}(y)$:
\begin{eqnarray*}
\| \ln |{\cal T}_{n}(x)-{\cal T}_{n}(y)| \|_{L^{p'}(A_{2})}^{p'} &=& \int_{B(0,2(1+2C_{1}))\setminus B(0,1+1/(4C_{1}))} |\ln |{\cal T}_{n}(x)-\xi||^{p'} |\det D{\cal T}_{n}^{-1}(\xi)|\, d\xi\\
&\leq&C \int_{B(0,3(1+2C_{1}))} |\ln |\eta||^{p'} \, d\eta\leq C
\end{eqnarray*}
where we have used that $({\cal T}_{n}^{-1})'$ converges uniformly to $({\cal T}^{-1})'$ on $B(0,2(1+2C_{1}))\setminus B(0,1+1/(4C_{1}))$ (see Proposition \ref{Caratheo} and the Cauchy formula) and that $({\cal T}^{-1})'$ is smooth on this annulus.

Putting together all the estimates  gives that for all $n\geq N$ we have
\[
\| \tilde\psi_{n}^0(t,\cdot)\|_{L^\infty(K')} \leq C (\| \omega_{n}(t,\cdot)\|_{L^1}+\| \omega_{n}(t,\cdot)\|_{L^p})
\]
and \eqref{om est} allows us to end the proof.
\end{proof}

If there is only one obstacle, $\tilde\psi_{n}^0=\psi_{n}^0$ and the previous lemma gives a uniform estimate of $\psi_{n}^0$. Let us now extend this estimate  to the case of several obstacles ($k \geq 2$). We introduce
\[
\psi_{n}:= \psi^0_{n}- \tilde \psi_{n}^0,
\]
which verifies:
\begin{equation}\label{psi harmonic}
\Delta \psi_{n} =  0 \text{ in }\Omega_{n},\quad \psi_{n}\vert_{\partial \Omega_{n}}=-\tilde\psi_{n}^0\vert_{\partial \Omega_{n}},\quad \psi_{n}(x) =\cal{O}(1) \text{ as }x\to \infty.
\end{equation}
To derive a uniform control on $\psi^0_n$, we rely on the following maximum principle in unbounded domains:
\begin{lemma}\label{max princ}
Let  $\psi_n$ the harmonic function verifying \eqref{psi harmonic}. Then, 
\[
|\psi_n (t,x) | \leq \sup_{\partial \Omega_{n}} |\tilde\psi_{n}^0(t,\cdot) |, \quad \forall \, t, \: \forall x\in \Omega_{n}.
\]
\end{lemma}
\begin{proof}[Proof of Lemma \ref{max princ}] The idea is to use an inversion mapping in order to send the unbounded domain to a bounded one. Without loss of generality, let us assume that $B(0,\rho_n)\subset O_{n}^1$ with some $\rho_n >0$, then the inversion map $\mathbf{i}(x):=x/|x|^2$ maps $\Omega_{n}$ to a bounded domain $\widetilde \Omega_{n}$ included in $B(0,1/\rho_n)$. Such a function has some interesting properties which can be found e.g. in \cite[Lem. 3.7]{LLL}. In particular, from the properties of $\psi_n(t,\cdot)$ (namely, the limit at infinity and the harmonicity), one can check that (for any fixed $t$) the function $g_n := \psi_n(t,\cdot) \circ \mathbf{i}^{-1}$ verifies:
\[
\Delta g_n =  0 \text{ in }\widetilde\Omega_{n},\quad g_n \vert_{\partial \widetilde\Omega_{n}}=\psi_n(t,\cdot) \circ \mathbf{i}^{-1}\vert_{\partial \Omega_{n}}.
\]
Of course, a key point is that $\psi_n(t,\cdot)$ has a limit at infinity, so that after inversion, $g_n$ is continuous at zero. It is then harmonic in a vicinity of zero, for instance because it satisfies there the mean-value formula.  Finally, the lemma follows from the standard maximum principle:
\[
\sup_{\Omega_{n}} |\psi_n(t,\cdot)| = \sup_{\widetilde\Omega_{n}} |g_n|   = \sup_{\pa \widetilde\Omega_{n}} |g_n| =   \sup_{\partial \Omega_{n}} |\tilde\psi_{n}^0(t,\cdot)|.
\]
\end{proof}

By the Hausdorff convergence and the disjointness of the obstacles, there exists a compact set $K$ of $\R^2\setminus \cal{C}^1$ and $N$ such that  $\partial O_{n}^i \subset K$ for all $i=2,\dots,k$ and  $n\geq N$. Hence, by Lemma \ref{uniform BS} and $\tilde\psi_{n}^0\vert_{\partial O_{n}^1}=0$, we infer that  $\sup_{\partial \Omega_{n}} |\tilde\psi_{n}^0|$ is uniformly bounded. Therefore, for any compact subset $K'$ of $\Omega_{n}$, Lemmas \ref{uniform BS} and \ref{max princ} imply that
\begin{equation}\label{psi-0 uniform}
\psi_{n}^0  \text{ is bounded  uniformly in $n$, $t$ and }x\in K'.
\end{equation}

\subsubsection{$H^1$ estimates}
From \eqref{psi-0 uniform} we deduce   a local $H^1$ estimate  away from the boundary. Indeed, for any compact subset $K'$ of $\Omega_{n}$, we can combine \eqref{om est}, \eqref{psi-0 uniform} with elliptic regularity for the equation   $\Delta \psi_{n}^0=\omega_{n}$ in $K'$. We get that
\begin{equation*}
 \psi_{n}^0(t,\cdot)  \text{ belongs to $W^{2,p}(K')$  uniformly in $n$, $t$}.
\end{equation*}
which obviously gives  the local $H^1$ estimate  away from the boundary:
\begin{equation} \label{uniformboundpsin0}
\psi_{n}^0(t,\cdot) \text{ belongs to }H^1_{\loc}(\Omega)\text{ uniformly in $n$ and $t$.}
\end{equation}

Then, let us show a uniform $H^1$ bound near the boundary. As for the harmonic function, we   use a smooth function $\chi$  with compact support, such that  $\chi = 1$ in a neighborhood of all obstacles ${\cal C}^i$.  The functions $\chi \, \psi_n^0$ is compactly supported and satisfies 
$$ \Delta (\chi \, \psi^0_n) \: = \:  \chi \omega_{n} + 2  \na \chi \cdot \na \psi^0_n + (\Delta \chi) \, \psi^0_n, \quad \chi \, \psi_n^0 \vert_{\partial\Omega_{n}}=0,$$
from which we infer 
$$ \int_{\Omega_n} | \na (\chi \psi^0_n)|^2 \: \le \: \|\omega_{n}\|_{H^{-1}(\supp \chi)}\|  \chi \, \psi^0_n \|_{H^1}  + \| 2 \na \chi \cdot \na \psi^0_n + \Delta \chi \, \psi^0_n  \|_{L^2 }  \|  \chi \, \psi^0_n \|_{L^2}. $$
For $n$ large enough, $\nabla \chi$ is supported away from the boundary of $\Omega_n$.  Therefore, we can use \eqref{om est},  \eqref{uniformboundpsin0} and Lemma \ref{Poincare} to conclude that 
$$ \| \na (\chi \psi^0_n)(t,\cdot) \|_{L^2(\Omega_{n})}  \: \le \: C  .$$
Still by Lemma \ref{Poincare}, we deduce a uniform control of $\psi^0_n$ in $H^1$ in a vicinity of the obstacles. Combining with  \eqref{uniformboundpsin0}, we get that 
$$ \psi_n^0(t,\cdot) \text{ is bounded  uniformly in $n$ and $t$ in } H^1_{\loc}(\overline{\Omega}),$$  
hence extending by zero
$$ \psi_n^0(t,\cdot) \text{ is bounded  uniformly in $n$ and $t$ in } H^1_{\loc}(\R^2).$$

\section{Compactness} \label{sect 4}
By the uniform estimates of the previous section, we can assume up to a subsequence that  
$$ \psi^i_n \rightharpoonup \psi^i \textrm{ weakly in  }  H^1_{\loc}(\R^2), \quad i=0,\dots,k $$
and
$$ \psi^0_n \rightharpoonup \psi^0 \textrm{ weakly-* in  } L^\infty(\R_+, H^1_{\loc}(\R^2)). $$

Here, we have implicitly extended the streamfunctions $\psi^i_n$ inside all obstacles $O^j_n$ by their (constant) values at $\pa O^j_n$. 
Also, extending $\omega_n$ by zero, we can assume that
$$ \omega_n \rightharpoonup \omega \text{ weakly-* in  } L^\infty(\R_+, L^1 \cap L^p(\R^2)).   $$ 
As a by-product, we can  obtain the convergence of the $\alpha^i_n$ weakly in $L^\infty(\R_{+})$,  see \eqref{defalphain}-\eqref{circu}:
\begin{equation} \label{eq_alphain}
 \alpha^i_n  \rightharpoonup \gamma^i(u^0) + \int_{\R^2} \chi^i \omega + \int _{\R^2}  \na \psi^0 \cdot \nabla \chi^i  =:\alpha^i.
\end{equation}
 where we have used assumption c) in Theorem \ref{theo stability}. Finally, back to \eqref{BS}, we obtain the weak-* convergence of $u_n$ to a limit field $u$ in $L^\infty(\R_+, L^2_{\loc}(\R^2))$.  This vector field has a  decomposition of the type:
$$
u=\nabla^\perp \psi^0 + \sum_{i=1}^k \alpha^i \nabla^\perp \psi^i.
$$
It is clear that
\[
\div u =0 \quad\text{and}\quad \curl u = \omega \quad \text{in}\quad \cal{D}'(\R_{+} \times \Omega).
\]

\medskip
To conclude the proof of Theorem \ref{theo stability}, we still need: 
\begin{itemize}
\item to prove that  $u$ satisfies \eqref{imperm2}.
\item to prove convergence of the approximate initial data $u^0_n$ to $u^0$. 
\item to prove that $u$ satisfies the momentum equation \eqref{Eulerweak}. 
\end{itemize}

\begin{remark} If we consider initial vorticities such that $\| \omega^0_{n} \|_{L^p(\Omega_{n})}\leq \| \omega^0\|_{L^p(\Omega)}$ (see \eqref{approx om}), then taking the liminf of the relation $\|\omega_{n}\|_{L^\infty(\R^+,L^q(\R^2))} = \|\omega_{n}^0\|_{L^q(\R^2)}$ we get \eqref{finalenergy}.
\end{remark}

\medskip
\subsection{Tangency condition} \label{subsection_tangency}
Let $h \in H^1_{\loc}(\Omega)$, such that $\na h \in L^2(\Omega)$ and $h(x) = 0$ for large $x$. We must prove that $\int_{\Omega}u \cdot \na h = 0$. Let $B$ be a ball containing the obstacles and the support  of $h$. We shall prove that 
\begin{enumerate}
\item For any $\chi \in {\cal D(\R^2)}$ with $\chi = 1$ near all obstacles, for almost all $t$
$$ \chi \, \psi^0(t,\cdot) \in H^1_0(\Omega).$$ 
\item There exist constants $C^{i,j}$, $i,j=1,\dots,k$ such that for any  $\chi^j \in {\cal D}(\R^2)$ with $\chi^j= 1$ near ${\cal C}^j$ and $\chi^j = 0$ near the other obstacles, 
$$ \chi^j (\psi^i - C^{i,j}) \in H^1_0(\Omega).$$
\end{enumerate}
These statements easily imply the result. Indeed, take $\chi = 1$ near $B$, $\chi^j$ such that $\sum \chi^j = 1$ near $B$. Then, we can write 
$$ u \: = \: \na^\perp  \left( \chi \psi^0  + \sum_{i,j} \alpha^i \, \chi^j (\psi^i - C^{i,j}) \right) + \sum_{i,j} \alpha^i C^{i,j} \na^\perp \chi^j   \quad \text{ over } B. $$
 Note that the vector field in the second term of the r.h.s. vanishes identically near the obstacles.  Moreover,   the first term of the r.h.s. belongs to $H^1_0(\Omega)$ for a.e. $t$. For these times $t$, we introduce  $\psi_n \in {\cal D}(\Omega)$ such that 
 $$ \psi_n \: \rightarrow  \: \chi \psi^0  + \sum_{i,j} \alpha^i \, \chi^j (\psi^i - C^{i,j})  \quad \text{ in } H^1(\Omega) $$
 and defining 
 $$v_n := \na^\perp \psi_n  +  \sum_{i,j}  \alpha^i C^{i,j}  \na^\perp \chi^j $$
 we get  $u = \lim v_n $ in $L^2(B)$, where $v_n$ is divergence-free and zero near all obstacles. Finally, 
 $$ \int_\Omega u \cdot \na h = \lim_{n \rightarrow +\infty} \int_\Omega v_n \cdot \na h = 0$$
 by a standard integration by parts.   Let us now indicate how to prove properties (1) and (2).
 
 \medskip
 (1) The proof relies on the notion of $\gamma$-convergence of open sets. It was already the key ingredient in our former paper \cite{GV-L}, and we refer to Appendix C in this paper for a reminder.  Let $\Omega'$ a big open disk such that $\text{supp} \, \chi$ and all obstacles are included in 
$\Omega'$. As the number of connected components of $\displaystyle \R^2 \setminus (\Omega' \cap \Omega_n)$  remains constant, 
 $$ \Omega' \cap \Omega_n \: \text{ $\gamma$-converges to } \: \Omega' \cap \Omega, \quad \text{as $n \rightarrow +\infty$}.$$
 Let now $\varphi = \varphi(t) \in L^1(\R_+)$. The function 
$$ f^\varphi_n  \: :=  \:  \int_{\R_+} \varphi(s) \chi  \psi^0_n(s,\cdot) ds$$
 belongs to $H^1_0(\Omega' \cap \Omega_n),$
and its extension by zero converges weakly in $H^1(\R^2)$ to 
$$ f^\varphi \: := \:  \int_{\R_+} \varphi(s)  \chi \psi^0(s,\cdot) ds. $$
 By the previous $\gamma$-convergence property, it follows that $f^\varphi \in H^1_0(\Omega' \cap \Omega)$ and so $f^\varphi \in H^1_0(\Omega)$ (remember that $\chi$ is supported in $\Omega'$). Taking $\varphi  = \frac{1}{2r}1_{[t-r, t+ r ]}$ for any $0 < r < t$, we find that 
 $$ F(t) \: := \: \frac{1}{2r} \, \int_{t-r}^{t + r}  \chi  \psi^0(s,\cdot) ds  \:  \in \: H^1_0(\Omega) $$
for any $0 < r < t$. For all Lebesgue points $t$ of $s \mapsto  \psi^0(s,\cdot)$ (in particular for a.e. $t$)   we get:
 $$ \chi \psi^0(t,\cdot) = \lim_{r \rightarrow 0}     \frac{1}{2r} \, \int_{t-r}^{t + r}   \chi  \psi^0(s,\cdot) ds  \: \in H^1_0(\Omega). $$

\medskip
(2) We know that $\psi^i_n$ is constant at $\pa O^j_n$, and denote this constant  by $C^{i,j}_n$.

\medskip
We claim that $C^{i,j}_n$  is bounded in $n$. Indeed, let $\Omega'$ a bounded open set such that $\text{supp} \, \chi^j \Subset \Omega'$, disjoint from the obstacles $\displaystyle O^k_n$, $k\neq j$.   The stream function $\psi^i_n$ satisfies the Dirichlet problem
$$ \Delta \psi^i_n = 0 \: \text{ in } \Omega' \setminus O^j_n, \quad \psi^i_n\vert_{\pa O^j_n} = C^{i,j}_n, \quad \psi^i_n\vert_{\pa \Omega'} = \varphi^i_n $$
where $\varphi^i_n$ ($ \: := \: \psi^i_n\vert_{\pa \Omega'}$) can be seen as a given data, uniformly bounded in $H^{1/2}(\pa \Omega')$ (because $\psi^i_n$ is bounded  in $\displaystyle W^{2-1/p,p}(\pa \Omega')$). Then, we can write 
$$ \psi^i_n \: = C^{i,j}_n \phi^n \: + \: \psi^n $$
where 
$ \phi^n$ and $\psi^n$ are defined as the solutions of the following systems: 
$$ \Delta \phi^n = 0 \: \text{ in } \Omega' \setminus O^j_n, \quad \phi^n\vert_{\pa O^j_n} = 1, \quad \phi^n\vert_{\pa \Omega'} = 0 $$
and 
$$ \Delta \psi^n = 0 \: \text{ in } \Omega' \setminus O^j_n, \quad \psi^n\vert_{\pa O^j_n} = 0, \quad \psi^n\vert_{\pa \Omega'} = \varphi^i_n.  $$
A standard energy estimate yields that $\psi^n$ is bounded in $H^1(\Omega')$ (after extension by zero in $O^j_n$). Hence, $\displaystyle  C^{i,j}_n \phi^n = \psi^i_n - \psi^n$ is also bounded in $H^1(\Omega')$. In particular, the sequence of real numbers $C^{i,j}_n \| \na \phi^n \|_{L^2(\Omega')}$ is bounded uniformly in $n$.  Finally, arguing along the lines of \cite[pages 139-141]{GV-L}, one can show that 
$$ \liminf_{n \rightarrow +\infty}  \| \na \phi^n \|_{L^2(\Omega')}  > 0 $$
(this positive lower limit being related to the positive capacity of the obstacle ${\cal C}^j$). Thus, $C^{i,j}_n$ is bounded in $n$. 

\medskip
Then, clearly, $\chi^j (\psi^i_n - C^{i,j}_n)$ is bounded in $H^1_0(\Omega' \cap \Omega_n)$ and up to a subsequence, its extension by zero  weakly converges to $\chi^j (\psi^i - C^{i,j})$ in $H^1_{\loc}(\R^2)$. We conclude as in (1) (it is even simpler here, as there is no time dependence).

\subsection{Convergence of the initial data.}\label{sect conv u0}
The bounds and weak convergence results that we have described for $u_n(t,\cdot)$ are in particular true at $t=0$. Thus, up to  a subsequence, $u^0_n \rightharpoonup \tilde u^0$ weakly in $L^2_{\loc}(\R^2)$, where $\tilde u^0$ has the form: 
\begin{equation} \label{formuletildeu0}
 \tilde u^0(x) \: := \: \na^\perp \psi^{0,0}(x) \: + \:  \sum_{i=1}^k \alpha^{i,0} \na^\perp \psi^i(x).  
 \end{equation}
The goal of this subsection is to prove that $\tilde u^0 = u^0$. 
We remind that, up to subsequences:
$$ \psi^{0,0} = \lim_{n \rightarrow +\infty} \psi^0_n\vert_{t=0}, \quad \psi^i = \lim_{n \rightarrow +\infty} \psi^i_n, \quad i \ge 1, $$
 where convergence holds weakly in $H^1_{\loc}$. Moreover, as in \eqref{eq_alphain} we have
 $$ \alpha^{i,0} \: = \: \gamma^i(u^0) + \int_{\R^2} \chi^i \omega^{0} + \int _{\R^2}  \na \psi^{0,0} \cdot \nabla \chi^i.
  $$

 \medskip
We claim that $\gamma^i(\tilde u^0) = \gamma^i(u^0)$. First, we recall that $\psi^{0,0}\vert_{\pa \Omega} = 0$, in the sense of Subsection \ref{subsection_tangency}: 
$$ \chi \psi^{0,0} \in H^1_0(\Omega).$$
In particular, 
\begin{equation*}  
 \alpha^{i,0} \: = \: \gamma^i(u^0) + \int_{\Omega} \chi^i \omega^0 + \int _{\Omega}  \na \psi^{0,0} \cdot \nabla \chi^i = \gamma^i(u^0) - \gamma^i(\nabla^\perp \psi^{0,0}),
 \end{equation*}  
 where the weak circulation is defined in \eqref{weak circu}. 
From this identity and formula \eqref{formuletildeu0}, it is enough to show that 
\begin{equation} \label{circuharmonic}
\gamma^i(\na^\perp \psi^i) \: = \delta_{ij}, \quad \text{ that is } \:  \:    \int_\Omega \na \psi^i \cdot \na \chi^j =    - \delta_{ij}.
\end{equation}
With the same notations as in Subsection \ref{subsection_tangency},  we compute
\begin{equation*}
 \int_\Omega \na \psi^i \cdot \na \chi^j  = \int_\Omega \na \left( \psi^i - C^{i,j} \right) \cdot \na \chi^j  
   = \int_{\R^2} \na \left(\psi^i - C^{i,j} \right) \cdot \na \chi^j = \int_{\R^2} \na \psi^i  \cdot \na \chi^j. 
 \end{equation*}
 Meanwhile, with the same kind of computations 
 \begin{align*}
  \delta_{ij} & = \oint_{\pa O^j_n} \na^\perp \psi^i_n \cdot \tau = - \int_{\Omega_n} \na \psi^i_n \cdot \na \chi^j  \\
  & = - \int_{\R^2} \na \left(\psi^i_n - C^{i,j}_n \right) \cdot \na \chi^j = -   \int_{\R^2} \na \psi^i_n  \cdot \na \chi^j 
  \rightarrow -   \int_{\R^2} \na \psi^i  \cdot \na \chi^j \: \text{ as } \: n \rightarrow +\infty .
  \end{align*}
Combining the last equalities, we find \eqref{circuharmonic}. 

\medskip
We conclude that $\gamma(\tilde u^0) = \gamma(u^0)$. Moreover,  we have clearly
$$ \div \tilde u^0 =0 \quad\text{and}\quad \curl \tilde u^0 = \omega^{0} \quad \text{in}\quad \cal{D}'(\Omega).$$
Also, the reasoning of the previous subsection can be applied to show that  $\tilde u^0$ verifies the tangency condition \eqref{imperm}.

\medskip
To be able to conclude that $\tilde u^0 = u^0$, we still need to deal with the behavior of these vector fields  at infinity. By assumption \eqref{initialdata} we know that $u^0\to 0$ as $|x|\to +\infty$, but we do not have yet any information about $\tilde u^0$. Therefore, we establish the following lemma:
\begin{lemma}
There exists an harmonic function $h$ on $\Omega$ such that:
\[
\tilde u^0 = K_{\R^2}[\omega^0]+h \quad \text{and}\quad h(x)=\mathcal{O}(1/|x|) \ \text{when}\ x\to\infty,
\]
where $K_{\R^2}[\omega^0]$ is the usual Biot-Savart formula in the full plane: $K_{\R^2}[\omega^0]=\frac{x^\perp}{2\pi|x|^2}\ast\omega^0$.
\end{lemma}
\begin{proof}
As $\omega_{n}^0$ is compactly supported, it is obvious that $K_{\R^2}[\omega^0_{n}] =\mathcal{O}(1/|x|)$, depending on the size of the support of $\omega_{n}^0$. Moreover, still using the compact support of $\omega_n$, we have noticed in Remark \ref{rem infinity} that $u_{n}^0=\mathcal{O}(1/|x|)$. Therefore, $h_{n}:=u_{n}^0 - K_{\R^2}[\omega^0_{n}]$  behaves as $\mathcal{O}(1/|x|)$ at infinity. Moreover, it is a harmonic vector field on $\Omega_n$. This means that it is divergence-free and curl-free, or equivalently that  
$\displaystyle  x + i y \mapsto h_{n,1}(x,y) - i h_{n,2}(x,y)$ is holomorphic over $\Omega_n$.   Now, by assumption b) in Theorem  \ref{theo stability}, we know that there is $q\in (1,2)$ such that  $\|\omega_{n}^0-\omega^0\|_{L^q(\R^2)}\to 0$. So the Hardy-Littlewood-Sobolev theorem (see e.g. \cite[Theo. V.1]{Stein} with $\alpha=1$) implies that:
\[
\| K_{\R^2}[\omega_{n}^0-\omega^0] \|_{L^{q^*}(\R^2)} \leq C_{q} \|\omega_{n}^0-\omega^0\|_{L^{q}(\R^2)}\to 0,
\]
where $\frac1{q}=\frac1{q^*}+\frac12$. As $q^*$ belongs to $(2,\infty)$, we get that
\[
K_{\R^2}[\omega_{n}^0] \to K_{\R^2}[\omega^0] \text{ strongly in } L^2_{\loc}(\R^2).
\]
Therefore, $h_{n}$ converges weakly  in $L^2_{\loc}(\R^2)$ to $h:=\tilde u^0 -K_{\R^2}[\omega^0]$. From  the harmonicity of $h_n$ and $h$ (identifying these harmonic fields with their holomorphic counterparts),   standard application of the mean-value theorem yields  that $h_{n}\to h$ strongly in $C^0_{\loc}(\Omega)$. We can write a Laurent expansion for $h$, and as we can compute the coefficients of this expansion from the value of $h$ on any circle $\partial B(0,R)$ (with $R$ large enough), the strong limit implies that $h=\mathcal{O}(1/|x|)$. This ends the proof of the lemma.
\end{proof}

Notice from the previous proof that  $K_{\R^2}[\omega^0]-K_{\R^2}[\omega_{n}^0]$ belongs to $L^{q^*}(\R^2)$ for some $q^*>2$. For $n$ fixed, it is also clear that $K_{\R^2}[\omega_{n}^0]\in L^{q^*}(\R^2)$, hence  $K_{\R^2}[\omega^0]$ belongs to $L^{q^*}(\R^2)$. Eventually, 
$\tilde u^0$ is the sum of a function in $L^{q^*}(\R^2)$ and a function $h$ which goes to zero at infinity. 

\medskip
We can  now show that $\tilde u^0 = u^0$ As the difference $v^0:=\tilde u^0-u^0$ is curl free and has zero circulation, we infer from the same arguments as in \cite[page 145]{GV-L} that $v^0=\nabla p^0$ for some smooth $p^0$ inside $\Omega$. Moreover, $v^0$ is the sum of a function  which tends to zero at infinity plus a function in $L^{q^*}(\R^2)$. From  the  harmonicity  of $v^0$ and the mean value theorem, it is easy to prove that $v^0$ goes to zero at infinity. Again,  $\bar{v}:=v^0_{1}-iv^0_{2}$ is holomorphic and admits a Laurent expansion $\sum_{n\in \Z} c_{k} z^{-k}$ where $c_{k} = \frac{1}{2i\pi}\int_{\partial B(0,R)} \bar{v}z^{k-1}\, dz$. First, taking the limit $R\to \infty$ in $|c_{k}|\leq  \|\bar v\|_{L^\infty(\partial B(0,R))} R^k$, we infer that $c_{k}=0$ for any $k\leq 0$. Second, we use the standard formula 
$$
2i\pi c_{1}=\int_{\partial B(0,R)} (v^0_{1}-iv^0_{2})  \, dz =  \int_{ \partial B(0,R)} v^0 \cdot \tau \, ds  -i \int_{ \partial B(0,R)} v^0\cdot n \, ds.
$$
The first right hand side integral vanishes because $v^0$ is a gradient.  The second is also zero because of the tangency condition \eqref{imperm}.  Indeed let $\chi=\sum \chi^i$ a smooth cutoff function which is equal to $1$ in a neighborhood of $\partial \Omega$ and zero on $\partial B(0,R)$ (for $R$ large enough) then
\[
\int_{ \partial B(0,R)} v^0\cdot n \, ds = \int_{B(0,R)\cap\Omega} v^0\cdot \nabla (1-\chi)=  -\int_{\Omega} v^0\cdot \nabla \chi=0.
\]
Therefore, we have  $c_{1}=0$ and $v^0 = \mathcal{O}(1/|x|^2)$. In particular, it implies that $v^0$ belongs to 
\[
G(\Omega) := \{ w\in L^2(\Omega),\ w=\nabla h \text{ for some }h\in H^1_{\loc}(\Omega)\}.
\]
Eventually, we claim that $v^0 \in G(\Omega)^\perp$. Indeed, let $h \in G(\Omega)$, and $\chi \in C^\infty_c(\R^2)$ such that $\chi = 1$ on a big open  ball $B$ containing all obstacles. By condition  \eqref{imperm}, 
$$ \int_\Omega  v \cdot \na h \: = \: \int_\Omega  v \cdot \na (\chi h) \: + \: \int_{\Omega}   v \cdot \na ((1 - \chi) h) \: = \: \int_{B^c} v \cdot \na g $$
where $g := (1-\chi) h$ satisfies: $g \in H^1_{\loc}(B^c)$, $\na g \in L^2(B^c)$, $g = 0$ over $\pa B$  in the trace sense. We deduce from \cite[Theorem II.7.3, page 104]{galdi}  that there exists $g_n \in C^\infty_c(\overline{B}^c)$ such that $\na  g_n \rightarrow \na g$ in $L^2(B^c)$. From there: 
$$   \int_{B^c} v \cdot \na g = \lim_{n \rightarrow +\infty}  \int_{\Omega} v \cdot \na g_n = 0 $$
still by \eqref{imperm}. 

\medskip
Hence, $v^0 = 0$,  $\tilde u^0 = u^0$, so that $u^0_n \rightharpoonup u^0$ weakly in $L^2_{\loc}(\R^2)$. 

\subsection{Convergence in the momentum equation.} Again, the arguments of \cite[page 158]{GV-L} can be applied {\it stricto sensu}.  In short, we take some smooth domain $\Omega' \Subset \Omega$, and decompose 
$$ u_n \: = \: P_{\Omega'} u_n \: + \: \na q_n \quad \text{ in } \Omega' $$
where $P_\Omega'$ is the Leray projector over divergence-free fields, tangent to $\pa \Omega'$. As $u_n$ is already divergence free, the remaining term is the gradient of a harmonic function $q_n$. 

\medskip
Using the momentum equation for $u_n$, we then get easily a uniform bound on $\pa_t P_{\Omega'} u_n$ in some negative Sobolev space. The strong compactness of $P_{\Omega'} u_n$ follows by Aubin-Lions lemma. Eventually, to pass to the limit in the convective term $\div(u_n \otimes u_n)$, we must show that the annoying weak product $\div(\na q_n \otimes \na q_n)$ converges to $0$ when integrated against smooth divergence-free fields $\varphi$ with compact support in $\Omega'$. This convergence follows from the algebraic identity 
$$ \int_{\Omega'} \na q \otimes \na q : \na \varphi = -  \int_{\Omega'} \frac{1}{2} \na |\na q|^2 \cdot \varphi \: + \:  \Delta q \, \na q \cdot \varphi = 0 $$
valid for all harmonic functions $q$. For all details, see \cite[page 158]{GV-L}.

\bigskip

\noindent
 {\bf Acknowledgements.} The authors are partially supported by the Project ``Instabilities in Hydrodynamics'' financed by Paris city hall (program ``Emergences'') and the Fondation Sciences Math\'ematiques de Paris.{ The authors are also grateful to Thierry de Pauw for the short proof of Proposition \ref{prop:equivalence}.}

\def\cprime{$'$}

\end{document}